\documentclass[12pt, reqno]{amsart}
\usepackage{amsmath, amsthm, amscd, amsfonts, amssymb, graphicx, color}
\usepackage[mathscr]{eucal}
\usepackage[bookmarksnumbered, colorlinks, plainpages]{hyperref}
\hypersetup{colorlinks=true,linkcolor=red, anchorcolor=green, citecolor=cyan, urlcolor=red, filecolor=magenta, pdftoolbar=true}

\newtheorem{theorem}{Theorem}

\newtheorem{proposition}[theorem]{Proposition}
\newtheorem{corollary}[theorem]{Corollary}
\theoremstyle{remark}
\newtheorem{remark}{Remark}

\DeclareMathAlphabet\mathoo{U}{eur}{b}{n}

\DeclareMathOperator{\Real}{Re}
\DeclareMathOperator{\heaviside}{\eta}

\begin{document}
\setcounter{page}{1}

\title[An estimate of Green's function]{An estimate of Green's function\\ of the problem of bounded solutions\\ in the case of a triangular coefficient}

\author{V.~G. Kurbatov}
 \address{Department of Mathematical Physics,
Voronezh State University\\ 1, Universitetskaya Square, Voronezh 394018, Russia}
\email{\textcolor[rgb]{0.00,0.00,0.84}{kv51@inbox.ru}}

\author{I.V. Kurbatova}
 \address{Department of Software Development and Information Systems Administration,
Vo\-ro\-nezh State University\\ 1, Universitetskaya Square, Voronezh 394018, Russia}
\email{\textcolor[rgb]{0.00,0.00,0.84}{la\_soleil@bk.ru}}

\subjclass{Primary 65F60; Secondary 97N50}

\keywords{bounded solutions problem, Green's function, estimate, Schur decomposition}

\date{\today}

\begin{abstract}
An estimate of Green's function of the bounded solutions problem for the ordinary differential equation $x'(t)-Bx(t)=f(t)$ is proposed.
It is assumed that the matrix coefficient $B$ is triangular.
This estimate is a generalization of the estimate of  the matrix exponential proved by Ch.~F. Van~Loan.
\end{abstract}

\maketitle

\section*{Introduction}\label{s:Introduction}
Any square complex matrix $A$ can be represented in the triangular Schur form $A=Q^{-1}BQ$, where $B$ is triangular and $Q$ is unitary. Moreover, the Schur decomposition is calculated effectively and constitutes the preliminary stage of many algorithms connected with spectral theory. Thus, in numerical analysis, having a matrix $A$, one may assume that its triangular form is known. In its turn, any triangular matrix $B$ can be easily represented as the sum $B=D+N$ of a diagonal matrix $D$ and a strictly triangular matrix $N$. In~\cite{Van_Loan77} it is established the estimate
\begin{equation}\label{e:Loan est}
\lVert e^{At}\rVert=\lVert e^{Bt}\rVert\le e^{\alpha t}\sum_{k=0}^{n-1}\frac{\lVert Nt\rVert^k}{k!},\qquad t\ge0,\tag{$*$}
\end{equation}
where the matrix $A$ has the size $n\times n$ and $\alpha=\max\{\,\Real\lambda:\,\lambda\in\sigma(A)\,\}$. For other estimates of $e^{At}$, see, e.g.,~\cite{Bulgakov80:eng,Gil-LMA-93}, \cite[ch. 9, \S~1]{Kato:eng}. In this paper, we establish an estimate similar to~\eqref{e:Loan est} for Green's function of the bounded solutions problem (Theorem~\ref{t:fin}). Different estimates for Green's function are obtained in~\cite{Bulgakov89a:eng,Godunov86:SMZh:eng,Kurbatov-Kurbatova-QTDS18}.

We consider the differential equation
\begin{equation}\label{e:eq}
x'(t)-Ax(t)=f(t),\qquad t\in\mathbb R,\tag{$**$}
\end{equation}
where $A$ is a square complex matrix.
The \emph{bounded solutions problem} is the problem of finding a bounded solution~$x$ that corresponds to a bounded free term $f$. The bounded solutions problem is closely connected with the problem of exponential dichotomy of solutions. For the discussion of the bounded solutions problem from different points of view and related questions, see~\cite{Akhmerov-Kurbatov,BaskakovMS15:eng,Boichuk-Pokutnii,Chicone-Latushkin99,
Daletskii-Krein:eng,
Godunov94:ODE:eng,Hartman02:eng,Henry81:eng,
Massera-Schaffer:eng,
Pankov90:eng,Pechkurov12:eng
} and the references therein.

It is known (Theorem~\ref{t:Green}) that equation~\eqref{e:eq} has a unique bounded solution $x$ for any bounded continuous free term~$f$ if and only if the spectrum of the coefficient $A$ is disjoint from the imaginary axis. In this case, the solution can be represented in the form
\begin{equation*}
x(t)=\int_{-\infty}^{\infty}\mathcal G(A,s)f(t-s)\,ds.
\end{equation*}
The kernel $\mathcal G$ is called \emph{Green's function}.
Assuming that $A$ is triangular or a triangular form of $A$ is known, we establish an effective estimate of Green's function (Theorem~\ref{t:fin}). Estimates of Green's function are important, e.g., for numerical methods~\cite{Bulgakov80:eng,Bulgakov89a:eng,Godunov86:SMZh:eng,
Kurbatov-Kurbatova-CMAM18,Malyshev91:eng}.

\section{The definition of Green's function}\label{s:Green's function}
For $\lambda\in\mathbb C$ and $t\in\mathbb R$, we consider the functions
\begin{align*}
\exp_{t}^+(\lambda)&=\begin{cases}
e^{\lambda t}, & \text{if $t>0$},\\
0, & \text{if $t<0$},\end{cases}\\
\exp_{t}^-(\lambda)&=\begin{cases}
0, & \text{if $t>0$},\\
-e^{\lambda t}, & \text{if $t<0$},\end{cases}
\\
g_t(\lambda)&=\begin{cases}
\exp_{t}^-(\lambda), & \text{if $\Real\lambda>0$},\\
\exp_{t}^+(\lambda), & \text{if $\Real\lambda<0$}.
\end{cases}
\end{align*}
These functions are undefined for $t=0$. The function $g_t$ is also undefined for $\Real\lambda=0$. For any fixed $t\neq0$, all three functions are analytic on their domains.

Let $A$ be a complex matrix of the size $n\times n$.
We consider the differential equation
\begin{equation}\label{e:x'=Ax+f}
x'(t)=Ax(t)+f(t),\qquad t\in\mathbb R.
\end{equation}
The \emph{bounded solutions problem} is the problem of finding bounded solution $x:\,\mathbb R\to\mathbb C^n$ under the assumption that the free term $f:\,\mathbb R\to\mathbb C^n$ is a bounded function.

\begin{theorem}[{\rm\cite[Theorem 4.1, p.~81]{Daletskii-Krein:eng}}]\label{t:Green}
Let $A$ be a complex matrix of the size $n\times n$.
Equation~\eqref{e:x'=Ax+f} has a unique bounded on $\mathbb R$ solution $x$ for any bounded continuous function $f$ if and only if the spectrum $\sigma(A)$ of $A$ does not intersect the imaginary axis.
This solution admits the representation
\begin{equation*}
x(t)=\int_{-\infty}^\infty \mathcal G(A,s)f(t-s)\,ds,
\end{equation*}
where
\begin{equation*}
\mathcal G(A,t)=\frac1{2\pi i}\int_\Gamma g_t(\lambda)(\lambda\mathbf1-A)^{-1}\,d\lambda,\qquad t\neq0,
\end{equation*}
the contour $\Gamma$ encloses $\sigma(A)$, and $\mathbf1$ is the identity matrix.
\end{theorem}
The function $\mathcal G$ is called~\cite{Daletskii-Krein:eng} \emph{Green's function} of the bounded solutions problem for equation~\eqref{e:x'=Ax+f}.

\begin{corollary}\label{c:A=B}
Let $A=Q^{-1}BQ$, where $Q$ is unitary. Then
\begin{equation*}
\lVert\mathcal G(A,t)\rVert_{2\to2}=\lVert\mathcal G(B,t)\rVert_{2\to2},\qquad t\neq0,
\end{equation*}
where $\lVert\cdot\rVert_{2\to2}$ is the matrix norm induced by the norm $\lVert z\rVert_2=\sqrt{|z_1|^2+\dots+|z_n|^2}$ on $\mathbb C^n$.
\end{corollary}
\begin{proof}
It is well known that for any analytic function $f$
\begin{equation*}
f(A)=Q^{-1}f(B)Q.\qed
\end{equation*}
\renewcommand\qed{}
\end{proof}

\section{The estimate}\label{s:estimate}

\begin{proposition}\label{p:ini}
For any square matrices $A$ and $B$ of the same size with the spectrum disjoint from the imaginary axis
\begin{equation*}
\mathcal G(A,t)-\mathcal G(B,t)=\int_{-\infty}^{+\infty}\mathcal G(A,t)(A-B)\mathcal G(B,t)\,dt.
\end{equation*}
\end{proposition}
\begin{proof}
We organize the proof as a sequence of references. By~\cite[Corollary 47]{Kurbatov-Kurbatova-Oreshina},
\begin{equation*}
\mathcal G(A,t)-\mathcal G(B,t)=\frac1{2\pi i}\int_\Gamma g_t(\lambda)(\lambda\mathbf1-A)^{-1}(A-B)(\lambda\mathbf1-B)^{-1}\,d\lambda,
\end{equation*}
where the contour $\Gamma$ surrounds $\sigma(A)\cup\sigma(B)$. By~\cite[Theorem 45]{Kurbatov-Kurbatova-Oreshina},
\begin{multline*}
\frac1{2\pi i}\int_\Gamma g_t(\lambda)(\mathbf1-A)^{-1}(A-B)(\mathbf1-B)^{-1}\,d\lambda\\=
\frac1{(2\pi i)^2}\int_{\Gamma_1}\int_{\Gamma_2} g_t^{[1]}(\lambda,\mu)(\lambda\mathbf1-A)^{-1}(A-B)(\mu\mathbf1-B)^{-1}\,d\mu\,d\lambda,
\end{multline*}
where $g_t^{[1]}$ is the divided difference of the function $g_t$, $\Gamma_1$ surrounds $\sigma(A)$, and $\Gamma_2$ surrounds $\sigma(B)$. We recall that the \emph{divided difference}~\cite{Gelfond:eng,Jordan} of a function $f$ is the function
\begin{equation}\label{e:f[1]}
f^{[1]}(\lambda,\mu)= \begin{cases}
\frac{f(\lambda)-f(\mu)}{\lambda-\mu}, & \text{if $\lambda\neq\mu$},\\
f'(\lambda), & \text{if $\lambda=\mu$}.
 \end{cases}
\end{equation}
For more about $g_t^{[1]}$, see, e.g.,~\cite{Kurbatov-Kurbatova-CMAM18}.
By~\cite[Theorem 23]{Kurbatov-Kurbatova-triangular1},
\begin{multline*}
\frac1{(2\pi i)^2}\int_{\Gamma_1}\int_{\Gamma_2} g_t^{[1]}(\lambda,\mu)(\lambda\mathbf1-A)^{-1}(A-B)(\mu\mathbf1-B)^{-1}\,d\mu\,d\lambda\\
=\int_{-\infty}^{\infty}\mathcal G(A,s)(A-B)\mathcal G(B,t-s)\,ds.\qed
\end{multline*}
\renewcommand\qed{}
\end{proof}

The main result of this paper is the following theorem.

\begin{theorem}\label{t:fin}
Let a complex triangular matrix $B$ have the size $n\times n$ and be represented as the sum $B=D+N$ of a diagonal matrix $D$ and a strictly triangular matrix $N$. Let $\gamma^->0$ and $\gamma^+>0$ be chosen so that the strip $\{z\in\mathbb C:\,-\gamma^+<\Real z<\gamma^-\,\}$ is disjoint from the spectrum $\sigma(B)$ of the matrix $B$.
Let the norm on the space of matrices be induced by a norm on $\mathbb C^n$ and possesses the property $\lVert D\rVert\le\max_i|d_{ii}|$ for any diagonal matrix $D$.
Then
\begin{equation}\label{e:est}
\lVert \mathcal G(B,t)\rVert
\le\bigl(\heaviside(t)e^{-\gamma^-t}+\heaviside(-t)e^{\gamma^+t}\bigr)\sum_{k=0}^{n-1}\lVert N\rVert^{k}
\frac{|t|^{k}
\sqrt{\gamma|t|}\,e^{\gamma|t|/2} K_{-k-\frac{1}{2}}\left(\frac{1}{2}
\gamma|t|\right)}{\sqrt{\pi}k!},
\end{equation}
where $K$ is the modified Bessel function of the second kind~\cite{Watson95:eng}, $\heaviside$ is the Heaviside function
\begin{equation*}
\heaviside(t)=
\begin{cases}
1& \text{ if $t>0$}, \\
0& \text{ if $t<0$}
\end{cases}
\end{equation*}
and
\begin{equation*}
\gamma=\gamma^-+\gamma^+.
\end{equation*}
\end{theorem}

\begin{remark}\label{r:2->2}
With the help of Corollary~\ref{c:A=B}, in the case of the $\lVert\cdot\rVert_{2\to2}$ norm, Theorem~\ref{t:fin} can be applied to any matrix $A$ provided its triangular representation $A=Q^{-1}BQ$ is known.
\end{remark}

\begin{remark}\label{r:BesselK}
It can be shown that the function $t\mapsto\dfrac{t^{k}\sqrt{\gamma t}\,e^{\gamma t/2} K_{-k-\frac{1}{2}}\left(\frac{1}{2}
\gamma t\right)}{\sqrt{\pi}k!}$ is a polynomial of degree $k$ with the leading term $\dfrac{t^{k}}{k!}$. For example, if $B$ is normal, then $N=0$ and, thus,
\begin{equation*}
\sum_{k=0}^{n-1}\lVert N\rVert^{k}\frac{|t|^{k}
\sqrt{\gamma|t|}\,e^{\gamma|t|/2} K_{-k-\frac{1}{2}}\left(\frac{1}{2}
\gamma|t|\right)}{\sqrt{\pi}k!}=\lVert N\rVert^{0}
\frac{|t|^{0}
\sqrt{\gamma|t|}\,e^{\gamma|t|/2} K_{-\frac{1}{2}}\left(\frac{1}{2}
\gamma|t|\right)}{\sqrt{\pi}0!}=1.
\end{equation*}
\end{remark}

\begin{remark}\label{r:Vin-Gro}
Another similar estimate of Green's function is proved in~\cite{Kurbatov-Kurbatova-QTDS18}:
\begin{align*}
\Vert\mathcal G(A,t)\Vert&\le e^{-\gamma^-t}\sum_{j=0}^{m-1}
\sum_{i=0}^j\binom{l+i-1}{l-1}\frac{t^{j-i}}{(j-i)!}
\frac{(2\Vert A\Vert)^{l+j}}{\gamma^{l+i}},& t&>0,\\
\Vert\mathcal G(A,t)\Vert&\le e^{\gamma^+t}\sum_{j=0}^{l-1}
\sum_{i=0}^j\binom{m+i-1}{m-1}\frac{t^{j-i}}{(j-i)!}
\frac{(2\Vert A\Vert)^{m+j}}{\gamma^{m+i}},& t&<0.
\end{align*}
Here $m$ (respectively, $l$) is the number of eigenvalues of $A$ counted according to their algebraic multiplicities that lie in the open left (right) half-plane. In this estimate the maximal powers of $t$ are $m-1$ and $l-1$ whereas the highest power of $t$ in~\eqref{e:est} is $n-1$. On the other hand, estimate from~\cite{Kurbatov-Kurbatova-QTDS18} contains the factors $\lVert A\rVert^{k}$ instead of $\lVert N\rVert^{k}$ in~\eqref{e:est}, which may be essentially smaller; for example, $N=0$ if $A$ is normal.
\end{remark}

\begin{proof}[Proof of Theorem~\ref{t:fin}]
From Proposition~\ref{p:ini} it follows that
\begin{equation}\label{e:G(D+N):1}
\mathcal G(B,t)=\mathcal G(D+N,t)=\mathcal G(D,t)+\int_{-\infty}^{\infty}\mathcal G(D+N,s)N\mathcal G(D,t-s)\,ds.
\end{equation}
We substitute this representation into itself:
\begin{align*}
\mathcal G(B,t)&=\mathcal G(D,t)+\int_{-\infty}^{\infty}
\mathcal G(D,s_1)N\mathcal G(D,t-s_1)\,ds_1\\
&+
\int_{-\infty}^{\infty}\int_{-\infty}^{\infty}\mathcal G(D+N,s_2)N\mathcal G(D,s_1-s_2)
N\mathcal G(D,t-s_1)\,ds_2\,ds_1.
\end{align*}
Then we substitute~\eqref{e:G(D+N):1} again into the obtained formula several times:
\begin{equation}\label{e:G(B,t)=}
\begin{split}
\mathcal G(B,t)&=\mathcal G(D,t)+\int_{-\infty}^{\infty}
\mathcal G(D,s_1)N\mathcal G(D,t-s_1)\,ds_1\\
&+
\int_{-\infty}^{\infty}\int_{-\infty}^{\infty}\mathcal G(D,s_2)N\mathcal G(D,s_1-s_2)
N\mathcal G(D,t-s_1)\,ds_2\,ds_1\\
&+
\int_{-\infty}^{\infty}\int_{-\infty}^{\infty}\int_{-\infty}^{\infty}\mathcal G(D,s_3)N\mathcal G(D,s_2-s_3)\\
&\times N\mathcal G(D,s_1-s_2)
N\mathcal G(D,t-s_1)\,ds_3\,ds_2\,ds_1+\dots\\
&+
\int_{-\infty}^{\infty}\dots\int_{-\infty}^{\infty}\mathcal G(D,s_{n-1})N\mathcal G(D,s_{n-2}-s_{n-1})\dots\\
&\times N\mathcal G(D,s_1-s_2)
N\mathcal G(D,t-s_1)\,ds_{n-1}\dots\,ds_1.
 \end{split}
\end{equation}
The subsequent terms are zero because of the following reason.
Since the matrix $D$ is diagonal, the matrix function $\mathcal G(D,\cdot)$ is also diagonal. The matrix $N$ is strictly triangular. Therefore the diagonal of $N\mathcal G(\cdot\cdot)N\mathcal G(\cdot)$ closest to the main one is zero, and so on.

From representation~\eqref{e:G(B,t)=} it follows that
\begin{align*}
\lVert \mathcal G(B&,t)\rVert\le\lVert\mathcal G(D,t)\rVert+
\lVert N\rVert\int_{-\infty}^{\infty}
\lVert \mathcal G(D,s_1)\rVert\lVert \mathcal G(D,t-s_1)\rVert\,ds_1\\
&+
\lVert N\rVert^2
\int_{-\infty}^{\infty}\int_{-\infty}^{\infty}\lVert \mathcal G(D,s_2)\rVert
\lVert \mathcal G(D,s_1-s_2)\rVert
\lVert \mathcal G(D,t-s_1)\rVert\,ds_2\,ds_1+\dots\\
&+
\lVert N\rVert^{n-1}\int_{-\infty}^{\infty}\dots\int_{-\infty}^{\infty}\lVert \mathcal G(D,s_{n-1})\rVert
\lVert \mathcal G(D,s_{n-2}-s_{n-1})\rVert\dots\\
&\times \lVert \mathcal G(D,s_1-s_2)\rVert
\lVert \mathcal G(D,t-s_1)\rVert\,ds_{n-1}\dots\,ds_1.
\end{align*}

We set
\begin{equation*}
h(t)=
\begin{cases}
e^{-\gamma^-t} & \text{ for $t\ge0$ }, \\
e^{\gamma^+t} & \text{ for $t\le0$ }.
\end{cases}
\end{equation*}
Since the matrix $D$ is diagonal,
\begin{equation*}
\lVert\mathcal G(D,t)\rVert\le h(t).
\end{equation*}
Therefore,
\begin{align*}
\lVert \mathcal G(D&+N,t)\rVert\le h(t)+
\lVert N\rVert\int_{-\infty}^{\infty}h(s_1)h(t-s_1)\,ds_1\\
&+
\lVert N\rVert^2
\int_{-\infty}^{\infty}\int_{-\infty}^{\infty}h(s_2)h(s_1-s_2)h(t-s_1)\,ds_2\,ds_1+\dots\\
&+
\lVert N\rVert^{n-1}\int_{-\infty}^{\infty}\dots\int_{-\infty}^{\infty}
h(s_{n-1})h(s_{n-2}-s_{n-1})\dots h(t-s_1)\,ds_{n-1}\dots\,ds_1.
\end{align*}
We denote the $k$-fold convolution of $h$ with itself by $h^{*k}$:
\begin{equation*}
h^{*k}(t)=\int_{-\infty}^{\infty}\dots\int_{-\infty}^{\infty}
h(s_{k-1})h(s_{k-2}-s_{k-1})\dots h(t-s_1)\,ds_{k-1}\dots\,ds_1.
\end{equation*}
With this notation, the previous estimate takes the from
\begin{equation*}
\lVert\mathcal G(B,t)\rVert\le\sum_{k=0}^{n-1}\lVert N\rVert^{k}h^{*(k+1)}(t).
\end{equation*}

Next we calculate $h^{*k}(t)$. We do some calculations with the help of `Mathematica'~\cite{Wolfram}.

Clearly, the Fourier transform $\hat h$ of $h$ is
\begin{equation*}
\hat h(\omega)=\frac1{\gamma^-+i\omega}+\frac1{\gamma^+-i\omega}.
\end{equation*}
Therefore the Fourier transform $\widehat{h^{*k}}$ of $h^{*k}$ is
\begin{equation*}
\widehat{h^{*k}}(\omega)=h^k(\omega)=\Bigl(\frac1{\gamma^-+i\omega}+\frac1{\gamma^+-i\omega}\Bigr)^k
=\sum_{j=0}^k\binom{k}{j}\frac1{(\gamma^+-i\omega)^j}\frac1{(\gamma^-+i\omega)^{k-j}}.
\end{equation*}

Further we have
\begin{align*}
\frac1{(\gamma^+-i\omega)^j}\frac1{(\gamma^-+i\omega)^{k-j}}&=
\sum_{q=0}^{k-j-1}\binom{j+q-1}{q}\frac1{\gamma^{j+q}(\gamma^-+i\omega)^{k-j-q}}\\
&+\sum_{q=0}^{j-1}\binom{k-j+q-1}{q}\frac1{\gamma^{k-j+q}(\gamma^+-i\omega)^{j-q}}.
\end{align*}
Consequently,
\begin{align*}
\widehat{h^{*k}}(\omega)
&=\sum_{j=0}^k\binom{k}{j}\frac1{(\gamma^+-i\omega)^j}\frac1{(\gamma^-+i\omega)^{k-j}}\\
&=\sum_{j=0}^k\binom{k}{j}
\sum_{q=0}^{k-j-1}\binom{j+q-1}{q}\frac1{\gamma^{j+q}(\gamma^-+i\omega)^{k-j-q}}\\
&+\sum_{j=0}^k\binom{k}{j}\sum_{q=0}^{j-1}\binom{k-j+q-1}{q}\frac1{\gamma^{k-j+q}(\gamma^+-i\omega)^{j-q}}\\
\end{align*}

It is well known that the inverse Fourier transform takes the functions
\begin{equation*}
\omega\mapsto\frac1{(\gamma^-+i\omega)^{k-j-q}},\qquad
\omega\mapsto\frac1{(\gamma^+-i\omega)^{j-q}},
\end{equation*}
respectively, to the functions
\begin{equation*}
t\mapsto\heaviside(t)\,\frac{t^{k-j-q-1}}{(k-j-q-1)!}e^{-\gamma^-t},\qquad
t\mapsto-\heaviside(-t)\,\frac{(-1)^{j-q}t^{j-q-1}}{(j-q-1)!}e^{\gamma^+t}.
\end{equation*}
Hence
\begin{align*}
h^{*k}(t)&=e^{-\gamma^-t}\,\heaviside(t)\sum_{j=0}^k\binom{k}{j}
\sum_{q=0}^{k-j-1}\binom{j+q-1}{q}\frac1{\gamma^{j+q}}
\frac{t^{k-j-q-1}}{(k-j-q-1)!}\\
&-e^{\gamma^+t}\,\heaviside(-t)\sum_{j=0}^k\binom{k}{j}
\sum_{q=0}^{j-1}\binom{k-j+q-1}{q}\frac1{\gamma^{k-j+q}}
\frac{(-1)^{j-q}t^{j-q-1}}{(j-q-1)!}.
\end{align*}
We change the order of summation:
\begin{align*}
h^{*k}(t)
&=e^{-\gamma^-t}\,\heaviside(t)\sum_{p=0}^{k-1}\sum_{j=0}^{p}
\binom{k}{j}\binom{p-1}{p-j}\frac1{\gamma^{p}}
\frac{t^{k-p-1}}{(k-p-1)!}\\
&-e^{\gamma^+t}\,\heaviside(-t)\sum_{p=1}^k\sum_{q=0}^{k-p}
\binom{k}{p+q}\binom{k-p-1}{q}\frac1{\gamma^{k-p}}
\frac{(-1)^{p}t^{p-1}}{(p-1)!}.
\end{align*}
First, we calculate the internal sums (here we essentially use `Mathematica'):
\begin{align*}
h^{*k}(t)
&=e^{-\gamma^-t}\,\heaviside(t)\sum_{p=0}^{k-1}
\frac{t^{k-p-1}(k+p-1)!}{\gamma^{p}\,(k-1)!p!(k-p-1)!}\\
&-e^{\gamma^+t}\,\heaviside(-t)\sum_{p=1}^k
\frac{(-1)^{k-p} t^{k-p-1}(k+p-1)!
}{\gamma^{p}\,(k-1)!p!(k-p-1)!}.
\end{align*}
Then, by calculating the final sums (here we again essentially use `Mathematica'), we arrive at
\begin{equation}\label{e:est of h}
\begin{split}
h^{*k}(t)
&=e^{-\gamma^-t}\,\heaviside(t)\frac{t^{k-1}
\sqrt{\gamma t}\,e^{\frac{1}{2}\gamma t} K_{\frac{1}{2}-k}\left(\frac{1}{2}
\gamma t\right)}{\sqrt{\pi }(k-1)!}\\
&+e^{\gamma^+t}\,\heaviside(-t)(-1)^{k-1}\frac{t^{k-1}\sqrt{-\gamma t}\,e^{-\frac{1}{2}\gamma t} K_{\frac{1}{2}-k}\left(-\frac{1}{2}\gamma
   t\right)}{\sqrt{\pi }(k-1)!}\\
&=h(t)\frac{|t|^{k-1}
\sqrt{\gamma|t|}\,e^{\gamma|t|/2} K_{\frac{1}{2}-k}\left(\frac{1}{2}
\gamma|t|\right)}{\sqrt{\pi}(k-1)!}.\qed
\end{split}
\end{equation}
\renewcommand\qed{}
\end{proof}

\begin{remark}\label{r:limit}
When $\gamma^+\to+\infty$, we have $e^{\gamma^+t}\to0$ for $t<0$ and (which is calculated in `Mathematica')
\begin{equation*}
\lim_{\gamma^+\to+\infty}\frac{|t|^{k}
\sqrt{\gamma|t|}\,e^{\gamma|t|/2} K_{-k-\frac{1}{2}}\left(\frac{1}{2}
\gamma|t|\right)}{\sqrt{\pi}k!}=1.
\end{equation*}
Therefore estimate~\eqref{e:est} turns into the estimate from~\cite{Van_Loan77}:
\begin{equation*}
\lVert\mathcal G(B,t)\rVert\le\heaviside(t) e^{-\gamma^-t}\sum_{k=0}^{n-1}\frac{\lVert N\rVert^{k}t^{k}}{k!}.
\end{equation*}
\end{remark}

Sometimes $\lVert N^k\rVert$ may decrease faster than $\lVert N\rVert^k$. In such a case it may be more convenient to use the following variant of Theorem~\ref{t:fin}.



\begin{corollary}\label{c:5}
Let the norm on the space of matrices be induced by the norm $\lVert z\rVert_1=|z_1|+\dots+|z_n|$ or the norm $\lVert z\rVert_\infty=\max_i|z_i|$ on $\mathbb C^n$. Then under the assumptions of Theorem~\ref{t:fin}
\begin{equation*}
|\mathcal G(B,t)|
\le\bigl(\heaviside(t)e^{-\gamma^-t}+\heaviside(-t)e^{\gamma^+t}\bigr)\sum_{k=0}^{n-1} |N|^{k}
\frac{|t|^{k}
\sqrt{\gamma|t|}\,e^{\gamma|t|/2} K_{-k-\frac{1}{2}}\left(\frac{1}{2}
\gamma|t|\right)}{\sqrt{\pi}k!},
\end{equation*}
where $|A|$ is the matrix consisting of absolute values $|a_{ij}|$ of the entries $a_{ij}$ of a matrix $A$, and the inequality $A\le B$ for matrices means the corresponding inequalities for their entries.
\end{corollary}

\begin{remark}\label{r:Loan}
The idea of using the power $|N|^k$ was used in~\cite[Theorem 11.2.2]{Golub-Van_Loan96:eng} for the estimating of the general function of a matrix.
\end{remark}

\begin{proof}
Both the norms in question possess the evident properties: $|A+B|\le|A|+|B|$, $|A\cdot B|\le|A|\cdot|B|$, $\lVert A\rVert=\lVert |A|\rVert$, and $\lVert |A|\rVert\le\lVert|B|\rVert$ if $|A|\le|B|$. Hence, from~\eqref{e:G(B,t)=} we have
\begin{align*}
|\mathcal G(B,t)|&\le|\mathcal G(D,t)|+\int_{-\infty}^{\infty}
|\mathcal G(D,s_1)N\mathcal G(D,t-s_1)|\,ds_1\\
&+
\int_{-\infty}^{\infty}\int_{-\infty}^{\infty}|\mathcal G(D,s_2)N\mathcal G(D,s_1-s_2)
N\mathcal G(D,t-s_1)|\,ds_2\,ds_1\\
&+
\int_{-\infty}^{\infty}\dots\int_{-\infty}^{\infty}|\mathcal G(D,s_{n-1})N\mathcal G(D,s_{n-2}-s_{n-1})\dots\\
&\times N\mathcal G(D,s_1-s_2)
N\mathcal G(D,t-s_1)|\,ds_{n-1}\dots\,ds_1.
\end{align*}
Therefore,
\begin{align*}
|\mathcal G(B,t)|&\le h(t)\mathbf1+\int_{-\infty}^{\infty}
(h(s_1)\mathbf1)|N|(h(t-s_1)\mathbf1)\,ds_1\\
&+
\int_{-\infty}^{\infty}\int_{-\infty}^{\infty}(h(s_2)\mathbf1)|N|(h(s_1-s_2)\mathbf1)
|N|(h(t-s_1)\mathbf1)\,ds_2\,ds_1\\
&+
\int_{-\infty}^{\infty}\dots\int_{-\infty}^{\infty} (h(s_{n-1})\mathbf1)|N|(h(s_{n-2}-s_{n-1})\mathbf1)\dots\\
&\times |N|(h(s_1-s_2)\mathbf1)
|N|(h(t-s_1)\mathbf1)\,ds_{n-1}\dots\,ds_1.
\end{align*}
Or
\begin{align*}
|\mathcal G(B,t)|&\le h(t)\mathbf1+\int_{-\infty}^{\infty}h(s_1)h(t-s_1)
|N|\,ds_1\\
&+
\int_{-\infty}^{\infty}\int_{-\infty}^{\infty}h(s_2)h(s_1-s_2)h(t-s_1)
|N|^2\,ds_2\,ds_1\\
&+
\int_{-\infty}^{\infty}\dots\int_{-\infty}^{\infty}h(s_{n-1})(h(s_{n-2}-s_{n-1})\dots h(s_1-s_2)|N|^{n-1}\,ds_{n-1}\dots\,ds_1\\
&=\sum_{k=0}^{n-1}|N|^{k} h^{*(k+1)}(t).
\end{align*}
Now the proof follows from estimate~\eqref{e:est of h}.
\end{proof}

\providecommand{\bysame}{\leavevmode\hbox to3em{\hrulefill}\thinspace}
\providecommand{\MR}{\relax\ifhmode\unskip\space\fi MR }
\providecommand{\MRhref}[2]{%
  \href{http://www.ams.org/mathscinet-getitem?mr=#1}{#2}
}
\providecommand{\href}[2]{#2}

\end{document}